\documentclass[11pt]{amsart} 
\usepackage{amsmath,amssymb,latexsym,esint,cite,mathrsfs}
\usepackage{verbatim,wasysym}
\usepackage[left=2.8cm,right=2.8cm,top=2.7cm,bottom=2.7cm]{geometry}

\usepackage{microtype}
\usepackage{color,enumitem,graphicx}
\usepackage[colorlinks=true,urlcolor=blue, citecolor=red,linkcolor=blue,
linktocpage,pdfpagelabels, bookmarksnumbered,bookmarksopen]{hyperref}
\usepackage[hyperpageref]{backref}
\usepackage[english]{babel}
\usepackage{lipsum}

\newtheorem{lemma}{Lemma}[section]
\newtheorem{corollary}[lemma]{Corollary}

\newtheorem{theorem}[lemma]{Theorem}

\theoremstyle{definition}

\newtheorem{remark}[lemma]{Remark}
\newtheorem{definition}[lemma]{Definition}
\newtheorem*{acknow}{Acknowledgements}

\numberwithin{equation}{section}

\newcommand{\R}{\mathbb{R}}
\newcommand{\s}{\mathbb{S}}

\newcommand{\esssup}{\operatornamewithlimits{ess\,sup}}
\newcommand{\essinf}{\operatornamewithlimits{ess\,inf}}
\newcommand{\eps}{\varepsilon}

\renewcommand{\Gamma}{\varGamma}

\newcommand{\Sf}{{\mathbb S^{n-1}}}

\def\Xint#1{\mathchoice 
  {\XXint\displaystyle\textstyle{#1}}%
  {\XXint\textstyle\scriptstyle{#1}}%
  {\XXint\scriptstyle\scriptscriptstyle{#1}}%
  {\XXint\scriptscriptstyle\scriptscriptstyle{#1}}%
  \!\int}
\def\XXint#1#2#3{{\setbox0=\hbox{$#1{#2#3}{\int}$} 
  \vcenter{\hbox{$#2#3$}}\kern-.5\wd0}}
\def\dashint{\Xint-}

\title[Characterizations of variable exponent Sobolev spaces]{Nonlocal characterizations of \\ variable exponent Sobolev spaces}

\author[G.\ Ferrari]{Gianluca Ferrari}
\author[M.\ Squassina]{Marco Squassina}

\address[G.\ Ferrari]{Dipartimento di Matematica e Fisica
	\newline\indent
	Universit\`a Cattolica del Sacro Cuore
	\newline\indent
	Via dei Musei 41, Brescia, Italy}
\email{gianluca.ferrari03@icatt.it}

\address[M.\ Squassina]{Dipartimento di Matematica e Fisica
	\newline\indent
	Universit\`a Cattolica del Sacro Cuore
	\newline\indent
	Via dei Musei 41, Brescia, Italy}
\email{marco.squassina@unicatt.it}

\thanks{The second author is member
	of {\em Gruppo Nazionale per l'Analisi Ma\-te\-ma\-ti\-ca, la Probabilit\`a e le loro Applicazioni} (GNAMPA) 
	of the {\em Istituto Nazionale di Alta Matematica} (INdAM)}

\begin{document}

\begin{abstract}
We obtain some nonlocal characterizations for a class of variable exponent Sobolev spaces 
arising in nonlinear elasticity theory and in the theory of electrorheological fluids. We also get 
a singular limit formula extending Nguyen results to the anisotropic case.
\end{abstract}

\maketitle


\section{Introduction}
In the last twenty years, starting from the work by Bourgain, Brezis and Mironescu \cite{BBM1}, 
there has been a considerable effort in the literature to provide some useful nonlocal characterizations 
of functions in Sobolev spaces. The results in \cite{BBM1} are mainly for $W^{1,p}(\Omega)$ on a bounded domain $\Omega$, but they can be 
extended to the whole space setting. In particular, if $p\in(1,+\infty)$ and $u\in L^p(\R^n)$, then $u\in W^{1,p}(\R^n)$ if and only if
$$
\sup_{s\in (0,1)}(1-s)\int_{\R^n}\int_{\R^n}\frac{|u(x)-u(y)|^{p}}{|x-y|^{n+ps}}\,dx\,dy<+\infty,
$$ 
in which case
$$
\lim_{s\nearrow 1}(1-s)\int_{\R^n}\int_{\R^n}\frac{|u(x)-u(y)|^{p}}{|x-y|^{n+ps}}\,dx\,dy=K_{n,p} \int_{\R^n}|\nabla u|^p dx,
$$
where
\begin{equation}
\label{K}
K_{n,p} = \frac{1}{p} \int_{\mathbb{S}^{n-1}} \left|\,\omega\cdot\boldsymbol{e}\,\right|^p\,d\mathcal{H}^{n-1}(\omega), \quad  \boldsymbol{e} \in \s^{n-1}.
\end{equation}
Another nonlocal characterization  was obtained by Nguyen in \cite{nguyen06,NgSob2,Ng11} 
involving a nonhomogenous functional. More precisely,
if $p\in(1,+\infty)$ and $u\in L^p(\R^n)$, then $u\in W^{1,p}(\R^n)$ if and only if
$$
\sup_{\delta\in (0,1)}\underset{|u(x)-u(y)|>\delta}{\int_{\R^n}\int_{\R^n}} \frac{\delta^{p}}{|x-y|^{n+p}} \,dx\,dy<+\infty,
$$
in which case
$$
\lim_{\delta\searrow 0} \underset{|u(x)-u(y)|>\delta}{\int_{\R^n}\int_{\R^n}} \frac{\delta^{p}}{|x-y|^{n+p}} \, dx \, dy = K_{n,p} \int_{\R^n} |\nabla u|^{p} \, dx.
$$
Nonhomogeneous quantities like these appear in some new estimates for the topological degree investigated in \cite{BBNg1}. The limiting
case $p=1$ is related to BV functions but it is actually more delicate, see \cite{nguyen06,NgSob2}.  
On the other hand, differential equations and variational problems involving variable $p(x)$-growth conditions, and hence variable exponent Sobolev spaces $W^{1,p(\cdot)}\left(\R^n\right)$, arise
from nonlinear elasticity theory and electrorheological fluids, and have been the target of various investigations, especially in regularity theory and in nonlocal problems 
(see e.g.\ \cite{acerbming,acerbming2,ruz,DHHR}).

Let $p:\R^n\to[1,+\infty)$ be a measurable function. Let's define
$$
p^- := \essinf_{\R^n} p  \qquad \text{and} \qquad p^+ := \esssup_{\R^n} p.
$$
For $x\in \R^n$, we set
\begin{equation}
\label{K(x)}
K_{n,p(x)} := \frac{1}{p(x)} \int_{\mathbb{S}^{n-1}} \left|\omega\cdot\boldsymbol{e}\right|^{p(x)}\,d\mathcal{H}^{n-1}(\omega) ,\quad \boldsymbol{e}\in\mathbb{S}^{n-1}.
\end{equation}
We also set
$$
W^{1,p^{\pm}}\left(\R^n\right):=W^{1,p^{+}}\left(\R^n\right)\cap W^{1,p^{-}}\left(\R^n\right) .
$$
In this framework, in the spirit of the results of \cite{nguyen06}, we have the following 

\begin{theorem}
\label{ANF}
Let $1<p^- \le p^+<+\infty$ and let $u\in W^{1,p^{\pm}}\left(\R^n\right)$. Then
\begin{enumerate}
\item[(a)] there exists $C>0$, depending only on $n$ and $p^\pm$, such that for every $\delta>0$
\begin{equation}
\label{convergence}
\underset{|u(x)-u(y)|>\delta}{\int_{\R^n}\int_{\R^n}} \frac{\delta^{p(x)}}{|x-y|^{n+p(x)}} \,dx\,dy \le C \left(\|\nabla u\|_{L^{p^+}\left(\R^n\right)}^{p^+}+\|\nabla u\|_{L^{p^-}\left(\R^n\right)}^{p^-}\right);
\end{equation}
\item[(b)] we have 
$$
\lim_{\delta\to0} \underset{|u(x)-u(y)|>\delta}{\int_{\R^n}\int_{\R^n}} \frac{\delta^{p(x)}}{|x-y|^{n+p(x)}} \, dx \, dy = \int_{\R^n} K_{n,p(x)} |\nabla u(x)|^{p(x)} \, dx ;
$$
\item[(c)]  if $u \in L^{p(\cdot)}\left(\R^n\right)$ and 
$$
\sup_{0<\delta<1}  \underset{|u(x)-u(y)|>\delta}{\int_{\R^n}\int_{\R^n}} \frac{\delta^{p(x)}}{|x-y|^{n+p(x)}}\, dx \, dy <+\infty  ,
$$
then $u\in W^{1,p(\cdot)}\left(\R^n\right)$.
\end{enumerate}
\end{theorem}

%

Unfortunately, it has not been possible to provide a limit formula in the more general context of the Sobolev space $W^{1,p(\cdot)}\left(\R^n\right)$,
the basic problem being that the Hardy-Littlewood maximal function in one direction on $L^{p(\cdot)}$ fails to be bounded (for the modular)
unless $p(\cdot)$ is a constant \cite{larsp}, see also \cite{Izuki}.
In the limit case $p^-=1$, let $u \in L^{p(\cdot)}(\R^n)$ and 
$E=\{x\in\R^n: p(x)=1\}.$  Assume ${\rm int}(E)\neq \emptyset$ and
$$
\sup_{0<\delta<1}  \underset{|u(x)-u(y)|>\delta}{\int_{\R^n}\int_{\R^n}} \frac{\delta^{p(x)}}{|x-y|^{n+p(x)}}\, dx \, dy <+\infty.
$$
Then, if $B$ denotes any ball in $\R^n$, it holds
$$
\sup_{B\subset {\rm int}(E)} |B|^{-\frac{n+1}{n}}\int_B\int_B |u(x)-u(y)|\,dx\,dy<+\infty,
$$
which, for $n=1$, reads as
$$
\sup_{B\subset {\rm int}(E)} \dashint_B\dashint_B |u(x)-u(y)|\,dx\,dy<+\infty,
$$
namely $u\in BMO(E)$, the space of bounded mean oscillation functions on $E$. See Remark \ref{remp1}.

\begin{acknow}
	We would like to thank Hoai-Minh Nguyen for providing uselful remarks about some points in the paper. 
	We also would like to warmly thank Lars Diening for pointing out some observations about the failure of
	boundedness estimates of the Hardy-Littlewood maximal function in one direction in the variable exponent case.
	Part of this paper was written during the preparation of the first author master thesis in Brescia.
\end{acknow}

\section{Preliminary stuff}
\noindent
In this section we recall some basic properties of the variable exponent spaces, see \cite{BDR,DHHR,DH}.

\subsection{Variable exponents spaces}
Let $\Omega\subset\R^n$ be a measurable set and let $p :\R^n\to[1,+\infty)$ be a measurable function. We define $L^{p(\cdot)}\left(\Omega\right)$ as the space of measurable functions $u :\Omega\to\R$ with
$$
\rho_{p(\cdot)} (u) :=\int_\Omega |u(x)|^{p(x)}\,dx < +\infty,
$$
so, denoting by $M\left(\Omega\right)$ the space of measurable functions on the domain $\Omega$, we set
$$
L^{p\left(\cdot\right)}\left(\Omega\right) := \left\{u\in M\left(\Omega\right) : \rho_{p(\cdot)} (u)<+\infty\right\} .
$$
The function $p$ is called exponent of $L^{p(\cdot)}\left(\Omega\right)$, while $\rho_{p(\cdot)}(u)$ is the modular of $u$. 
If $u\in L^{p(\cdot)}\left(\Omega\right)$, 
$$
\|u\|_{L^{p(\cdot)}\left(\Omega\right)} := \inf\left\{\lambda>0 : \rho_{p(\cdot)} \left(\frac{u}{\lambda}\right)\le1\right\}
$$
is a norm for $L^{p(\cdot)} \left(\Omega\right)$, called Luxemburg norm, which makes the space complete. 
In other words, $L^{p(\cdot)} \left(\Omega\right)$ is a Banach space with respect to $\|\cdot\|_{L^{p(\cdot)}\left(\Omega\right)}$.
Taken a locally integrable function $\omega : \R^n \to (0,+\infty)$, we can introduce a weighed version of variable exponent Lebesgue spaces. We define $L^{p(\cdot)}(\Omega,\omega)$ as the space of measurable functions $u : \Omega \to \R$ such that
$$
\rho_{p(\cdot),\,\omega}(u) := \int_{\Omega} |u(x)|^{p(x)} \, \omega(x) \, dx < +\infty,
$$
so we can set
$$
L^{p(\cdot)}(\Omega,\omega) := \left\{u\in M\left(\Omega\right) : \rho_{p(\cdot),\,\omega}(u) < +\infty \right\}.
$$
The function $\omega$ is called weight of the space. Moreover, $L^{p(\cdot)}(\Omega,\omega)$ is a Banach space with norm
$$
\|u\|_{L^{p(\cdot)}\left(\Omega,\omega\right)} := \inf \left\{\lambda > 0 : \rho_{p(\cdot),\,\omega}\left(\frac{u}{\lambda}\right)\le1\right\} .
$$
The following relationship between the norm and the modular holds:
\begin{equation}
\label{weight}
\min\left\{\rho_{p(\cdot),\,\omega}(u)^{\frac{1}{p^-}},\rho_{p(\cdot),\,\omega}(u)^{\frac{1}{p^+}}\right\}
\le \|u\|_{L^{p(\cdot)}\left(\Omega,\,\omega\right)}
\le \max\left\{\rho_{p(\cdot),\,\omega}(u)^{\frac{1}{p^-}},\rho_{p(\cdot),\,\omega}(u)^{\frac{1}{p^+}}\right\} .
\end{equation}
\noindent
We denote by $W^{1,p(\cdot)}\left(\Omega\right)$ the space of 
$u \in L^{p(\cdot)}\left(\Omega\right)$ such that their gradient $\nabla u\in L^{p(\cdot)}\left(\Omega\right)$, so that
$$
W^{1,p(\cdot)}\left(\Omega\right) := \left\{u\in L^{p(\cdot)}\left(\Omega\right) : \exists\nabla u\in L^{p(\cdot)}\left(\Omega\right)\right\} .
$$
If $u\in W^{1,p(\cdot)}\left(\Omega\right)$, the object
$$
\|u\|_{W^{1,p(\cdot)}\left(\Omega\right)} := \|u\|_{L^{p(\cdot)}\left(\Omega\right)} + \|\nabla u\|_{L^{p(\cdot)}\left(\Omega\right)}
$$
is a norm for $W^{1,p(\cdot)} \left(\Omega\right)$. Moreover, $W^{1,p(\cdot)}\left(\Omega\right)$ is a Banach space 
with respect to  $\|\cdot\|_{W^{1,p(\cdot)}\left(\Omega\right)}$.

\subsection{Fractional Sobolev spaces}
If $s\in(0,1)$, we can also extend the concept of fractional Sobolev space to the variable exponent case, as follows. 
Let $\Omega\subset\R^n$ be a measurable set and let $p:\R^n\times\R^n\to[1,+\infty)$ and $q:\R^n\to[1,+\infty)$ be two measurable functions. If we suppose that $p$ and $q$ are two bounded exponents, then there exist $p^+,p^-,q^+,q^-\in[1,+\infty)$ such that
$$
\forall x,y\in\R : \qquad p^-\le p(x,y)\le p^+, \qquad q^-\le q(x) \le q^+ .
$$
Taken $s\in(0,1)$, we denote by $W=W^{s,q(\cdot),p(\cdot\,,\,\cdot)}\left(\Omega\right)$ the functions space
$$
W := \left\{u\in L^{q(\cdot)}\left(\Omega\right)  : \exists\lambda>0 : \int_\Omega\int_\Omega \frac{|u(x)-u(y)|^{p(x,y)}}{\lambda^{p(x,y)}|x-y|^{n+sp(x,y)}}\,dx\,dy<+\infty\right\} .
$$
If we set the variable exponent seminorm as
$$
[u]_{s,p(\cdot\,,\,\cdot)} := \inf \left\{\lambda>0 : \int_\Omega\int_\Omega \frac{|u(x)-u(y)|^{p(x,y)}}{\lambda^{p(x,y)}|x-y|^{n+sp(x,y)}}\,dx\,dy\le 1 \right\} ,
$$
it is possible to prove that $W$ is a Banach space with respect to the norm
$$
\|u\|_W := \|u\|_{L^{q(\cdot)}\left(\Omega\right)}+[u]_{s,p(\cdot\,,\,\cdot)} .
$$
The concepts introduced are consistent with the classical definitions of the spaces $L^p\left(\Omega\right)$, $W^{1,p}\left(\Omega\right)$ and $W^{s,p}\left(\Omega\right)$ when the functions $p$ and $q$ are equal and constant.

\subsection{Maximal functions}
Let $u\in L^1_{\rm loc}\left(\R^n\right)$ be a local summable function. We define its maximal function $\mathcal{M}(u)$ by setting
$$
\mathcal{M}(u)(x):=\sup_{r>0} \, \dashint_{B(x,r)} |u(y)| \, dy =\sup_{r>0} \frac{1}{\mathcal{L}^n\left(B(x,r)\right)} \int_{B(x,r)} |u(y)| \, dy ,
$$
where $\mathcal{L}^n\left(A\right)$ represents the $n$-dimensional Lebesgue measure of $A\subset\R^n$.
Moreover, we introduce the Hardy-Littlewood maximal operator as the function
$\mathcal{M} :\left\{u\mapsto \mathcal{M}(u)\right\}$.

\begin{theorem}
\label{teomax}
Let $p\in(1,+\infty]$. Then there exists a constant $C>0$, depending only on the dimension $n$ of the space and on the index $p$, such that
\begin{equation*}
\forall u\in L^{p}\left(\R^n\right) : \quad \|\mathcal{M}(u)\|_{L^{p}\left(\R^n\right)} \le C \, \|u\|_{L^{p}\left(\R^n\right)}.
\end{equation*}
In other words, the Hardy-Littlewood maximal operator $\mathcal{M} : L^{p}\left(\R^n\right)\to L^{p}\left(\R^n\right)$ is bounded.
\end{theorem}
\begin{proof}
See \cite[Chapter 1, Theorem 1]{Stein}.
\end{proof}

\noindent
For our purposes, taken any $\omega \in \Sf$, we also define
$$
\mathcal{M}_\omega(u)(x) := \sup_{h>0} \frac{1}{h} \int_0^h |u(x+s\omega)|\,ds \,
$$
as the maximal function of $u$ along the considered direction $\omega$. Arguing as in \cite[Lemma 3.1]{NPSV}, it is possible to prove the existence of a universal constant $C>0$ such that, for all $\omega \in \Sf$,
$$
\int_{\R^n} |\mathcal{M}_\omega(u)(x)|^p \,dx \le C \int_{\R^n} |u(x)|^p \,dx , \qquad \forall u \in L^p\left(\R^n\right).
$$ 
%
For variable exponents the inequality fails unless $p(\cdot)$ is constant \cite{larsp,Izuki}. For instance, if
$p(x)=2$ on $(-\infty,-2)$ and $p(x)\geq 4$ on $[2,+\infty)$, then $\int_\R |u|^{p(x)}dx<+\infty$, 
but  $\int_\R |\mathcal{M}(u)|^{p(x)}dx=+\infty$ for the function $u(x)=|x|^{-1/3}\chi_{[2,\infty)}(x)$.

\begin{definition}
A function $\alpha : \Omega\subset\R^n\to \R$ is called $\log$-H\" older continuous on $\Omega$ if there exists a constant $c>0$ such that
\begin{equation}
\label{holder}
\forall x,y\in \Omega: \quad |\alpha(x)-\alpha(y)| \le \frac{c}{\log\left(e+|x-y|^{-1}\right)} .
\end{equation}
Moreover, $\alpha$ satisfies the $\log$-H\" older decay condition if there exist $\alpha_\infty\in\R$ and $c>0$ such that
\begin{equation}
\label{decay}
\forall x\in \Omega: \quad |\alpha(x)-\alpha_\infty| \le \frac{c}{\log\left(e+|x|\right)} .
\end{equation}
The function $\alpha$ is called globally $\log$-H\" older continuous on the domain $\Omega$ if it is $\log$-H\" older continuous on $\Omega$ and it satisfies the decay condition just introduced. In this case, the constant $c$ satisfying both the equations \eqref{holder} and \eqref{decay} is called $\log$-H\" older constant of $\alpha$.
\end{definition}
\noindent
Now let's introduce the following class of variable exponents:
\begin{equation*}
\mathcal{P}^{\log}\left(\Omega\right) := \left\{p\in M\left(\Omega\right) : \frac{1}{p}\text{ is globally $\log$-H\" older continuous}\right\}.
\end{equation*}
We denote by $c_{\log}\left(p\right)$, or $c_{\log}$, the $\log$-H\" older constant of $1 \slash p$ and, if $\Omega$ is a bounded domain, we are able to introduce the index $p_\infty$ by setting
\begin{equation*}
p_\infty := \left(\lim_{|x|\to+\infty} \frac{1}{p(x)}\right)^{-1} ,
\end{equation*}
with the usual convention $1\slash+\infty=0$.

Let us introduce an important theorem about the boundness of the Hardy-Littlewood maximal operator $\mathcal{M} : \left\{u\mapsto \mathcal{M}\left(u\right)\right\}$ on the Lebesgue space $L^{p(\cdot)}\left(\R^n\right)$.

\begin{theorem}
\label{maxfunctiongen}
Let $p\in\mathcal{P}^{\log}\left(\R^n\right)$ be a bounded variable exponent, with $p^->1$. Then there exists a constant $K_{p^-}>0$, depending only on the dimension $n$ of the space and on the $\log$-H\" older constant $c_{\log}(p)$ of $1 \slash p$, such that
\begin{equation*}
\forall u\in L^{p(\cdot)}\left(\R^n\right) : \quad \|\mathcal{M}(u)\|_{L^{p(\cdot)}\left(\R^n\right)} \le K_{p^-} \|u\|_{L^{p(\cdot)}\left(\R^n\right)}.
\end{equation*}
In other words, the Hardy-Littlewood maximal operator $\mathcal{M} : L^{p(\cdot)}\left(\R^n\right)\to L^{p(\cdot)}\left(\R^n\right)$ is bounded.
\end{theorem}
\begin{proof}
See \cite[Theorem 4.3.8]{DHHR}.
\end{proof}
\noindent
As shown in \cite[Corollary 4.3.11]{DHHR}, the previous theorem holds also in the case of exponents  $p\in\mathcal{P}^{\log}\left(\Omega\right)$ and functions $u\in L^{p\left(\cdot\right)}\left(\Omega\right)$, with $\Omega\subset \R^n$.

\section{Anisotropic formulas}
\noindent

We are now ready to study the behavior of the singular limit in the anisotropic case. 
In the following, unless otherwise stated, we will assume $1<p^-\leq p^+<+\infty$.

\begin{lemma}
\label{lemma4.2}
Let $u\in W^{1,p^{\pm}}\left(\R^n\right)$.
Then there exists a positive constant $C$, dependent only on $n$ and $p^\pm,$  such that for all $\delta>0$
$$
\underset{|u(x)-u(y)|>\delta}{\int_{\R^n}\int_{\R^n}} \frac{\delta^{p(x)}}{|x-y|^{n+p(x)}}\,dx\,dy \le C \left(\|\nabla u\|_{L^{p^+}\left(\R^n\right)}^{p^+}+\|\nabla u\|_{L^{p^-}\left(\R^n\right)}^{p^-}\right) .
$$
In particular, the integral at the left-hand side is finite.
\end{lemma}
\begin{proof}
By using polar coordinates, we have
$$
\begin{aligned}
\underset{\left|u(x)-u(y)\right|>\delta} {\int_{\R^n}\int_{\R^n}} \frac{\delta^{p(x)}}{|x-y|^{n+p(x)}} \, dx \, dy & = \underset{\left|u(x+h\omega)-u(x)\right|>\delta} {\int_{\mathbb{S}^{n-1}}\int_{\R^n} \int_0^{+\infty}}\frac{\delta^{p(x)}}{h^{n+p(x)}} h^{n-1}  \,dh \, dx \, d\mathcal{H}^{n-1}(\omega) \\
& = \underset{\left|u(x+h\omega)-u(x)\right|>\delta} {\int_{\mathbb{S}^{n-1}}\int_{\R^n} \int_0^{+\infty}}\frac{\delta^{p(x)}}{h^{p(x)+1}} \,dh \, dx \, d\mathcal{H}^{n-1}(\omega) .
\end{aligned}
$$
Thanks to this equation, it is sufficient to prove the existence of a constant $C>0$, dependent only on $p^\pm$, such that for all $\omega\in\mathbb{S}^{n-1}$ we have
$$
\underset{|u(x+h\omega)-u(x)|>\delta}{\int_{\R^n}\int_0^{+\infty}} \frac{\delta^{p(x)}}{h^{p(x)+1}}\,dh\,dx \le C \left(\|\nabla u\|_{L^{p^+}\left(\R^n\right)}^{p^+}+\|\nabla u\|_{L^{p^-}\left(\R^n\right)}^{p^-}\right) .
$$
From the fundamental theorem of calculus,
$$
\begin{aligned}
\left|u\right(x+h\omega \left)-u(x)\right| & \le \int_{0}^{h} \left| \frac{d}{ds} u(x+s\omega) \right| \, ds = \int_{0}^{h} \left| \nabla u(x+s\omega) \cdot \omega \right| \, ds \\
& \le \int_{0}^{h} \left| \nabla u(x+s\omega) \right| \, ds \le h\, \mathcal{M}_\omega \left(\nabla u\right)(x) ,
\end{aligned}
$$
for a.e. $(x,h)\in\R^n\times(0,+\infty)$, where
$$
\mathcal{M}_\omega \left(u\right)\left(x\right) = \sup_{h>0} \, \dashint_{0}^{h} \left|u\left(x+s\omega\right)\right| \, ds
$$
is the maximal function of $u$ with respect to the direction $\omega\in\Sf$ previously introduced. From this inequality, it follows the set inclusion
$$
\left\{x\in\R^n : \left|u\left(x+h\omega\right)-u(x)\right|>\delta\right\}\subset\left\{x\in\R^n : h \, \mathcal{M}_\omega \left(\nabla u\right)(x)>\delta\right\},
$$
from which we have
$$
\begin{aligned}
\underset{\left|u\left(x+h\omega\right)-u(x)\right|>\delta}{\int_{\R^n}\int_0^{+\infty}} \frac{\delta^{p(x)}}{h^{p(x)+1}} \, dh\,dx & \le \int_{\R^n} \underset{h \mathcal{M}_\omega\left(\nabla u\right)(x) >\delta}{\int_0^{+\infty}} \frac{\delta^{p(x)}}{h^{p(x)+1}} \, dh\,dx \\
& = \int_{\R^n} \left[-\frac{1}{p(x)}\frac{\delta^{p(x)}}{h^{p(x)}}\right]_{{\delta}\slash{\mathcal{M}_\omega\left(\nabla u\right)(x)}}^{+\infty} \,dx \\
& = \int_{\R^n} \frac{1}{p(x)} \left|\mathcal{M}_\omega\left(\nabla u\right)(x)\right|^{p(x)} \,dx .
\end{aligned}
$$
Recalling that $p^-\le p(x)\le p^+$, we can increase the multiplicative inverse of $p(x)$ to $1 \slash p^-$, getting
\begin{equation}
\label{keyineq}
\underset{\left|u\left(x+h\boldsymbol{e}_n\right)-u(x)\right|>\delta}{\int_{\R^n}\int_0^{+\infty}} \frac{\delta^{p(x)}}{h^{p(x)+1}} \, dh\,dx 
\le \frac{1}{p^-} \int_{\R^n} \left|\mathcal{M}_\omega\left(\nabla u\right)(x)\right|^{p(x)} \,dx .
\end{equation}
At this point, we split the integral at the right-hand side, over the sets of $x\in\R^n$ with
$$
\mathcal{M}_\omega\left(\nabla u\right)(x)\le1 \qquad \text{or} \qquad \mathcal{M}_\omega\left(\nabla u\right)(x)>1 ,
$$
so that we are able to increase the integrand function by using, respectively, the exponents $p^-$ or $p^+$ and then by extending both the integrals over the entire space $\R^n$. In this way, we get
$$
\frac{1}{p^-}  \int_{\R^n} \left|\mathcal{M}_\omega\left(\nabla u\right)(x)\right|^{p(x)} \,dx 
\le \frac{1}{p^-} \int_{\R^n} \left|\mathcal{M}_\omega\left(\nabla u\right)(x)\right|^{p^-} \,dx
+ \frac{1}{p^-} \int_{\R^n} \left|\mathcal{M}_\omega\left(\nabla u\right)(x)\right|^{p^+} \,dx .
$$
As a direct conseguence of the theory of maximal funcrions, there exist positive constants $C_{p^\pm}$, depending on $p^\pm$, such that
$$
\begin{aligned}
\frac{1}{p^-} \int_{\R^n} \left|\mathcal{M}_\omega\left(\nabla u\right)(x)\right|^{p(x)} \,dx
& \le \frac{1}{p^-} \int_{\R^{n}} \left|\mathcal{M}_\omega \left(\nabla u\right)(x)\right|^{p^-} \,dx +\frac{1}{p^-} \int_{\R^{n}} \left|\mathcal{M}_\omega \left(\nabla u\right)(x)\right|^{p^+} \,dx \\
&\le \frac{C_{p^-}}{p^-} \int_{\R^{n}} \left|\nabla u(x)\right|^{p^-} \,dx + \frac{C_{p^+}}{p^-} \int_{\R^{n}} \left|\nabla u(x)\right|^{p^+} \,dx \\
&\le C \left(\|\nabla u\|_{L^{p^+}\left(\R^n\right)}^{p^+}+\|\nabla u\|_{L^{p^-}\left(\R^n\right)}^{p^-}\right) ,
\end{aligned}
$$
where $C:=\max\left\{C_{p^+},C_{p^-}\right\}/ p^-$. The assertion follows.
\end{proof}

\begin{remark}
Observe that, if $p : \R^n\to[1,+\infty)$ is bounded and measurable, then we have
$$
W^{1,p^{\pm}}\left(\R^n\right) \subset  W^{1,p(\cdot)}\left(\R^n\right) .
$$
In fact, taken $u\in W^{1,p^{\pm}}\left(\R^n\right)$, first of all
$$
\begin{aligned}
\int_{\R^n} \left|u(x)\right|^{p(x)} dx & = \int_{u(x)\le1} \left|u(x)\right|^{p(x)} dx + \int_{u(x)>1} \left|u(x)\right|^{p(x)} dx \\
& \le \|u\|_{L^{p^-}\left(\R^n\right)}^{p^-}+ \|u\|_{L^{p^+}\left(\R^n\right)}^{p^+} 
\end{aligned}
$$
and, at the same time,
$$
\int_{\R^n} \left|\nabla u(x)\right|^{p(x)} dx \le \|\nabla u\|_{L^{p^-}\left(\R^n\right)}^{p^-}+ \|\nabla u\|_{L^{p^+}\left(\R^n\right)}^{p^+} ,
$$
so $u\in W^{1,p(\cdot)}\left(\R^n\right)$.
\end{remark}

\begin{theorem}[Anisotropic limit I]
\label{teorema4.3.1}
For all $u\in W^{1,p^{\pm}}\left(\R^n\right)$, we have the limit formula
$$
\begin{aligned}
\lim_{\delta\to0} \underset{\left|u(x)-u(y)\right|>\delta}{\int_{\R^n} \int_{\R^n}} \frac{\delta^{p(x)}}{|x-y|^{n+p(x)}} & \, dx \, dy = \int_{\R^n} K_{n,p(x)} \left|\nabla u(x) \right|^{p(x)} \, dx ,
\end{aligned}
$$
where $K_{n,p(x)}$ is defined as in \eqref{K(x)}. In particular, the limit exists and is finite.
\end{theorem}
\begin{proof}
First of all, taken $h>0$, let's prove that, for all $\omega \in \mathbb{S}^{n-1}$, we have
\begin{equation}
\label{eq:dadimnuova}
\sup_{\delta\in (0,1)}\underset{\left|\frac{u(x+\delta h\omega)-u(x)}{\delta h}\right|h>1}{\int_{\R^n}\int_0^{+\infty}} \frac{1}{h^{p(x)+1}} \, dh\,dx <+\infty , 
\end{equation}
and
\begin{equation}
\label{eq:dadim2nuova}
\lim_{\delta\to0} \underset{\left|\frac{u(x+\delta h\omega)-u(x)}{\delta h}\right|h>1}{\int_{\R^n}\int_0^{+\infty}} \frac{1}{h^{p(x)+1}} \, dh\,dx = \int_{\R^n} \frac{1}{p(x)}\left|\nabla u(x)\cdot \omega \right|^{p(x)} \, dx .
\end{equation}
Since $u\in W^{1,p^{\pm}}\left(\R\right)$,
$$
u(x+h\omega)-u(x) = \int_{0}^{h} \frac{d}{ds} u(x+s\omega)\,ds = \int_{0}^{h} \left(\nabla u(x+s\omega)\cdot\omega\right)\,ds
$$
for all $(x,h)\in\R^n\times(0,+\infty)$.
Let's set
$$
\begin{aligned}
A(\delta)&:=\left\{(x,h)\in\R^n\times(0,+\infty) : \left|\frac{u(x+\delta h \omega)-u(x)}{\delta h}\right|h>1\right\} , \\
A&:=\left\{(x,h)\in\R^n\times(0,+\infty) : \Biggl|\left. \frac{d}{ds}u(x+s\omega)\right|_{s=0} \Biggr|h>1\right\} , \\
B&:=\Bigl\{(x,h)\in\R^n\times(0,+\infty) : \mathcal{M}_\omega\left(\nabla u\right)(x) \, h>1\Bigr\} ,
\end{aligned}
$$
and let $\chi_K\left(x,h\right)$ be the characteristic function of a set $K\subset \R^n\times(0,+\infty)$. 
By the definition of maximal function, we have the inequalities chain

$$
\begin{aligned}
\mathcal{M}_\omega \left(\nabla u\right)(x) & = \sup_{h>0} \, \frac{1}{h} \int_{0}^{h} \left|\nabla u(x+s\omega) \right| \, ds \\
& \ge \sup_{h>0} \, \frac{1}{h}\left| \int_{0}^{h} \frac{d}{ds} u(x+s\omega) \, ds \right| \\
& = \sup_{h>0} \, \frac{1}{h}\left| u(x+h\omega)-u(x)\right| \ge \left|\frac{u(x+\delta h\omega)-u(x)}{\delta h}\right| , \quad \forall \delta >0 ,
\end{aligned}
$$
so that
$$
A(\delta)\subset B \Longrightarrow \chi_{A(\delta)}\left(x,h\right)\le\chi_{B}\left(x,h\right), \qquad \forall \left(x,h\right)\in\R^n\times(0,+\infty).
$$
Since
$$
\int_{\R^n} \int_0^{+\infty} \frac{1}{h^{p(x)+1}} \chi_{B}\left(x,h\right)\,dh\,dx = \int_{\R^n} \frac{1}{p(x)}\left|\mathcal{M}_\omega\left(\nabla u\right)(x)\right|^{p(x)} \, dx ,
$$
from what observed in the proof of the previous Lemma, the right hand side is finite, so we have
$$
\int_{\R^n} \int_0^{+\infty} \frac{1}{h^{p(x)+1}} \chi_{A\left(\delta\right)} \left(x,h\right) \,dh\,dx 
\le \int_{\R^n} \int_0^{+\infty} \frac{1}{h^{p(x)+1}} \chi_{B}\left(x,h\right)\,dh\,dx<+\infty ,
$$
as well as \eqref{eq:dadimnuova}. 
The function $\dfrac{\chi_{A\left(\delta\right)}\left(x,h\right)}{h^{p(x)+1}}$ is dominated by $\dfrac{\chi_{B}\left(x,h\right)}{h^{p(x)+1}}$, that is summable. Since
$$
\lim_{\delta \to 0} \chi_{A(\delta)}\left(x,h\right) =  \chi_{A}\left(x,h\right)
$$
for a.e. $\left(x,h\right)\in\R^n\times\R\times(0,+\infty)$, from dominated convergence, it follows 
$$
\begin{aligned}
\lim_{\delta\to0} & \underset{\left|\frac{u(x+\delta h\omega)-u(x)}{\delta h}\right|h>1}{\int_{\R^n}\int_0^{+\infty}} \frac{1}{h^{p(x)+1}} \, dh\,dx = \int_{\R^n} \int_0^{+\infty} \frac{1}{h^{p(x)+1}} \chi_{A}\left(x,h\right)\,dh\,dx \\
&= \int_{\R^n} \frac{1}{p(x)}\Biggl|\left. \frac{d}{ds}u(x+s\omega)\right|_{s=0} \Biggr|^{p(x)} \, dx = \int_{\R^n} \frac{1}{p(x)} \left|\nabla u(x)\cdot \omega\right|^{p(x)} \, dx.
\end{aligned}
$$
So, since \eqref{eq:dadim2nuova} holds, we are ready to prove the Lemma. 
Making the change of variable $y-x=\delta h \omega$, with $\omega\in\mathbb{S}^{n-1}$, we have
$$
\begin{aligned}
\underset{\left|u(x)-u(y)\right|>\delta}{\int_{\R^n} \int_{\R^n}} \frac{\delta^{p(x)}}{|x-y|^{n+p(x)}} \, dx \, dy & = \underset{\left|\frac{u(x+\delta h\omega)-u(x)}{\delta h} \right|h>1}{\int_{\R^n} \int_{\mathbb{S}^{n-1}}\int_0^{+\infty}} \frac{\delta^{p(x)}}{|\delta h |^{n+p(x)}} \,\delta^n h^{n-1}\, dh\,d\mathcal{H}^{n-1}(\omega)\,dx \\
& =\underset{\left|\frac{u(x+\delta h\omega)-u(x)}{\delta h} \right|h>1}{\int_{\R^n} \int_{\mathbb{S}^{n-1}}\int_0^{+\infty}} \frac{1}{| h |^{p(x)+1}} \, dh\,d\mathcal{H}^{n-1}(\omega)\,dx ,
\end{aligned}
$$
so, by using dominated convergence theorem and equation \eqref{eq:dadim2nuova}, it follows
$$
\begin{aligned}
\lim_{\delta\to0}\underset{\left|u(x)-u(y)\right|>\delta}{\int_{\R^n} \int_{\R^n}} \frac{\delta^{p(x)}}{|x-y|^{n+p}} \, dx \, dy &= \lim_{\delta\to 0}\underset{\left|\frac{u(x+\delta h\omega)-u(x)}{\delta h} \right|h>1}{\int_{\mathbb{S}^{n-1}} \int_{\R^n}\int_0^{+\infty}} \frac{1}{h^{p(x)+1}} \,dh\,dx\,d\mathcal{H}^{n-1}(\omega)\\
& = \int_{\mathbb{S}^{n-1}} \int_{\R^n} \frac{1}{p(x)} \left|\nabla u(x)\cdot \omega \right|^{p(x)} \, dx\,d\mathcal{H}^{n-1}(\omega) .
\end{aligned}
$$
Since, for every $V\in\R^n$ and for all $p\ge1$, we have
$$
\int_{\mathbb{S}^{n-1}} \left|V\cdot\omega\right|^p \, d\mathcal{H}^{n-1}(\omega) = p\,K_{n,p} \, |V|^p ,
$$
where $K_{n,p}$ is defined as in \eqref{K}, then, for every $V\in\R^n$ and for all $x\in\R^n$, we have
$$
\int_{\mathbb{S}^{n-1}} \left|V\cdot\omega\right|^{p(x)} \, d\mathcal{H}^{n-1}(\omega) = p(x)\,K_{n,p(x)} \, |V|^{p(x)} .
$$
As a consequence, we are able to obtain the Nguyen type limit formula
$$
\lim_{\delta\to0} \underset{\left|u(x)-u(y)\right|>\delta}{\int_{\R^n} \int_{\R^n}} \frac{\delta^{p(x)}}{|x-y|^{n+p(x)}} \, dx \, dy = \int_{\R^n} K_{n,p(x)} \left|\nabla u(x) \right|^{p(x)} \, dx .
$$
Observe that, by the definition of $K_{n,p(x)}$, we have
$$
\begin{aligned}
\lim_{\delta\to0} \underset{\left|u(x)-u(y)\right|>\delta}{\int_{\R^n} \int_{\R^n}} \frac{\delta^{p(x)}}{|x-y|^{n+p(x)}} \, dx \, dy & = \int_{\R^n} \frac{1}{p(x)} \left(\int_{\mathbb{S}^{n-1}}{\left|\omega\cdot\boldsymbol{e}\right|^{p(x)}}\,d\mathcal{H}^{n-1}(\omega)\right)\left|\nabla u(x) \right|^{p(x)} \, dx \\
& \le \frac{1}{p^{-}} \int_{\mathbb{S}^{n-1}}{\left|\omega\cdot\boldsymbol{e}\right|^{p^-}}\,d\mathcal{H}^{n-1}(\omega)\int_{\R^n} \left|\nabla u(x) \right|^{p(x)} \, dx \\
& = K_{n,p^-}\int_{\R^n} \left|\nabla u(x) \right|^{p(x)} \, dx,
\end{aligned}
$$
so that the limit is finite.
\end{proof}


\begin{remark}
The function $\left\{x\mapsto K_{n,p(x)}\right\}$ is also bounded, since $|K_{n,p(x)}|\le K_{n,p^-}$.
Moreover $\left\{s\mapsto K_{n,s}\right\}$, for $s\in[1,+\infty)$, is a monotonically decreasing function, that tends to $0$ as $s\to+\infty$. 
In fact, we have
$$
\begin{aligned}
\frac{d}{ds}\left[\frac{1}{s} \int_{\s^{n-1}} |\omega\cdot\boldsymbol{e}|^s\,d\mathcal{H}^{n-1}(\omega)\right] 
& = -\frac{1}{s^2} \int_{\s^{n-1}} |\omega\cdot\boldsymbol{e}|^s\,d\mathcal{H}^{n-1}(\omega) \\
& \quad + \frac{1}{s} \int_{\s^{n-1}} |\omega\cdot\boldsymbol{e}|^{s} \log{|\omega\cdot\boldsymbol{e}|}\,d\mathcal{H}^{n-1}(\omega) < 0 , \qquad  s\geq 1 ,
\end{aligned}
$$
and the assertions follow. Hence, roughy speaking, very large anisotropic exponents $p(x)$ will produce,
in some sense, a small measure  $\mu=K_{n,p(x)}{\mathcal L}^n$ in the limit formula.
\end{remark}

In the classical case, where $p$ is a constant exponent, if we have
$$
\lim_{\delta\to0} \underset{\left|u(x)-u(y)\right|>\delta}{\int_{\R^n} \int_{\R^n}} \frac{\delta^p}{|x-y|^{n+p}} \, dx \, dy = K_{n,p} \int_{\R^n} \left|\nabla u(x) \right|^p \, dx ,
$$
then we have also
$$
\lim_{\delta\to0} \underset{\left|u(x)-u(y)\right|>\delta}{\int_{\R^n} \int_{\R^n}} \frac{p\,\delta^p}{|x-y|^{n+p}} \, dx \, dy = p\, K_{n,p} \int_{\R^n} \left|\nabla u(x) \right|^p \, dx .
$$
This fact is not obvious in the variable exponent case.

\begin{theorem}[Anisotropic limit II]
\label{teorema4.3.2}
For all $u\in W^{1,p^{\pm}}\left(\R^n\right)$, we have the limit formula
$$
\begin{aligned}
\lim_{\delta\to0} \underset{\left|u(x)-u(y)\right|>\delta}{\int_{\R^n} \int_{\R^n}} \frac{p(x)\,\delta^{p(x)}}{|x-y|^{n+p(x)}} & \, dx \, dy = \int_{\R^n} p(x)\,K_{n,p(x)} \left|\nabla u(x) \right|^{p(x)} \, dx ,
\end{aligned}
$$
where $K_{n,p(x)}$ has been introduced previously. In particular, the limit exists and is finite.
\end{theorem}
\begin{proof}
In analogy to the previous case, taken $h>0$, we prove that, for all $\omega \in \mathbb{S}^{n-1}$,
\begin{equation}
\label{eq:dadimnuova2}
\sup_{\delta\in (0,1)}\underset{\left|\frac{u(x+\delta h\omega)-u(x)}{\delta h}\right|h>1}{\int_{\R^n}\int_0^{+\infty}} \frac{p(x)}{h^{p(x)+1}} \, dh\,dx <+\infty
\end{equation}
and
\begin{equation}
\label{eq:dadim2nuova2}
\lim_{\delta\to0} \underset{\left|\frac{u(x+\delta h\omega)-u(x)}{\delta h}\right|h>1}{\int_{\R^n}\int_0^{+\infty}} \frac{p(x)}{h^{p(x)+1}} \, dh\,dx = \int_{\R^n} \left|\nabla u(x)\cdot \omega \right|^{p(x)} \, dx .
\end{equation}
Of course, we have the inequality
$$
\underset{\left|\frac{u(x+\delta h\omega)-u(x)}{\delta h}\right|h>1}{\int_{\R^n}\int_0^{+\infty}} \frac{p(x)}{h^{p(x)+1}} \, dh\,dx \le p^+ \underset{\left|\frac{u(x+\delta h\omega)-u(x)}{\delta h}\right|h>1}{\int_{\R^n}\int_0^{+\infty}} \frac{1}{h^{p(x)+1}} \, dh\,dx
$$
and, by what previously proved, \eqref{eq:dadimnuova2} holds.
Now, taking into account the equations obtained for the first anisotropic limit formula, we have
$$
\int_{\R^n} \int_0^{+\infty} \frac{p(x)}{h^{p(x)+1}} \chi_{B}\left(x,h\right)\,dh\,dx = \int_{\R^n} \left|\mathcal{M}_\omega\left(\nabla u\right)(x)\right|^{p(x)} \, dx,
$$
where the right-hand side is convergent, as proved in Lemma \ref{lemma4.2}. 
From $A(\delta)\subset B$, we get
$$
\begin{aligned}
\frac{p(x)}{h^{p(x)+1}} \chi_{A\left(\delta\right)}\left(x,h\right) \le \frac{p(x)}{h^{p(x)+1}} \chi_{B}\left(x,h\right) , \qquad \forall (x,h)\in\R^n\times(0,+\infty) ,\quad\forall \delta>0,
\end{aligned}
$$
where
$p(x)h^{-p(x)-1} \chi_{B}\left(x,h\right)$ is summable. 
Then, by the dominated convergence theorem, 
$$
\begin{aligned}
\lim_{\delta\to0} \underset{\left|\frac{u(x+\delta h\omega)-u(x)}{\delta h}\right|h>1}{\int_{\R^n}\int_0^{+\infty}} \frac{p(x)}{h^{p(x)+1}} \, dh\,dx & = \int_{\R^n} \int_0^{+\infty} \frac{p(x)}{h^{p(x)+1}} \chi_{A}\left(x,h\right)\,dh\,dx = \int_{\R^n} \left|\nabla u(x)\cdot \omega\right|^{p(x)} \, dx ,
\end{aligned}
$$
proving \eqref{eq:dadim2nuova2}. Now we are ready to prove the theorem. 
Setting $y-x=\delta h \omega$, with $\omega\in\mathbb{S}^{n-1}$, 
$$
\underset{\left|u(x)-u(y)\right|>\delta}{\int_{\R^n} \int_{\R^n}} \frac{p(x)\,\delta^{p(x)}}{|x-y|^{n+p(x)}} \, dx \, dy = \underset{\left|\frac{u(x+\delta h\omega)-u(x)}{\delta h} \right|h>1}{\int_{\R^n} \int_{\mathbb{S}^{n-1}}\int_0^{+\infty}} \frac{p(x)}{h^{p(x)+1}} \, dh\,d\mathcal{H}^{n-1}(\omega)\,dx ,
$$
so, making the limit as $\delta\to0$, it results
$$
\begin{aligned}
\lim_{\delta\to0}\underset{\left|u(x)-u(y)\right|>\delta}{\int_{\R^n} \int_{\R^n}} \frac{p(x)\,\delta^{p(x)}}{|x-y|^{n+p}} \, dx \, dy & = \lim_{\delta\to 0}\underset{\left|\frac{u(x+\delta h\omega)-u(x)}{\delta h} \right|h>1}{\int_{\mathbb{S}^{n-1}} \int_{\R^n}\int_0^{+\infty}} \frac{p(x)}{h^{p(x)+1}} \, dh\,dx\,d\mathcal{H}^{n-1}(\omega) \\
& = \int_{\mathbb{S}^{n-1}} \int_{\R^n} \left|\nabla u(x)\cdot \omega \right|^{p(x)} \, dx\,d\mathcal{H}^{n-1}(\omega) .
\end{aligned}
$$
In conclusion, applying Fubini-Tonelli's theorem and recalling that, for any $V\in\R^n$ and  $x\in\R^n$
$$
\int_{\mathbb{S}^{n-1}} \left|V\cdot\omega\right|^{p(x)} \, d\mathcal{H}^{n-1}(\omega) = p(x)\,K_{n,p(x)} \, |V|^{p(x)} ,
$$
we get
$$
\lim_{\delta\to0} \underset{\left|u(x)-u(y)\right|>\delta}{\int_{\R^n} \int_{\R^n}} \frac{p(x)\,\delta^{p(x)}}{|x-y|^{n+p(x)}} \, dx \, dy = \int_{\R^n} p(x)\, K_{n,p(x)} \left|\nabla u(x) \right|^{p(x)} \, dx .
$$
Of course the right-hand side is finite,
so that the limit is also finite. The proof is complete.
\end{proof}

\section{Sufficient conditions}

First we state the following

\begin{lemma}
\label{lemma4.4}
Let $u\in L^{p(\cdot)}\left(\R^n\right)\cap C^2\left(\R^n\right)$. Then we have
$$
\int_{\R^n} p(x)\,K_{n,p(x)} \left|\nabla u(x)\right|^{p(x)} \, dx \le \liminf_{\varepsilon\to0} \int_{\R^n} \int_{\R^n} \frac{\eps\left|u(x)-u(y)\right|^{p(x)+\eps}}{|x-y|^{n+p(x)}} \, dx \, dy .
$$
Moreover, if $u$ satisfies
$$
C(u) := \sup_{0<\eps<1} \int_{\R^n} \int_{\R^n} \frac{\eps\left|u(x)-u(y)\right|^{p(x)+\eps}}{|x-y|^{n+p(x)}}\,dx \, dy <+\infty ,
$$
then $u\in W^{1,p(\cdot)}\left(\R^n\right)$.
\end{lemma}
\begin{proof}
By using polar coordinates, we get
$$
C(u) = \sup_{0<\eps<1} \int_{\mathbb{S}^{n-1}} \int_{\R^n}\int_0^{+\infty} \frac{\eps\left|u(x+r\omega)-u(x)\right|^{p(x)+\eps}}{r^{p(x)+1}}\,dr\,dx\,d\mathcal{H}^{n-1}(\omega).
$$
Consider the restriction to the open balls $B_A\subset \R^n$ at the origin with radius $A>0$,
$$
\sup_{0<\eps<1} \int_{\mathbb{S}^{n-1}} \int_{B_A}\int_0^{+\infty} \frac{\eps\left|u(x+r\omega)-u(x)\right|^{p(x)+\eps}}{r^{p(x)+1}}\,dr\,dx\,d\mathcal{H}^{n-1}(\omega) \le C(u) .
$$
Since $u\in C^2\left(\R^n\right)$, arguing as in the proof of \cite[Lemma 4]{nguyen06}, 
we have
$$
\left|Du(x)\cdot r\omega\right|^{p(x)+\eps} \le |u(x+r\omega)-u(x)|^{p(x)+\eps} + Cr^{p(x)+\eps+1} , \qquad \forall\,(\omega,x,r)\in\mathbb{S}^{n-1}\times B_A \times(0,1) .
$$
Multiplying by $\eps$ and dividing by $r^{p(x)+1}$, after integrating, we get
$$
\begin{aligned}
\liminf_{\varepsilon\to0} \int_{\s^{n-1}}\int_{B_A}\int_0^1 & \frac{ \eps\left|Du(x)\cdot r\omega\right|^{p(x)+\eps}}{r^{p(x)+1}}\, dr\,dx\,d\mathcal{H}^{n-1}(\omega) \\
 &\le \liminf_{\eps\to0} \int_{\s^{n-1}}\int_{B_A}\int_0^1\frac{\eps\,|u(x+r\omega)-u(x)|^{p(x)+\eps}}{r^{p(x)+1}}\,dr\,dx\,d\mathcal{H}^{n-1}(\omega) \\ 
&\quad+ C\lim_{\eps\to0} \int_{\s^{n-1}}\int_{B_A}\int_0^1 \eps \, r^{\eps}\,dr\,dx\,d\mathcal{H}^{n-1}(\omega) \\
& = \liminf_{\eps\to0} \int_{\s^{n-1}}\int_{B_A}\int_0^1\frac{\eps\,|u(x+r\omega)-u(x)|^{p(x)+\eps}}{r^{p(x)+1}}\,dr\,dx\,d\mathcal{H}^{n-1}(\omega).
\end{aligned}
$$
After some computation, we are able to apply Fatou's Lemma as shown in the following equation
$$
\begin{aligned}
\liminf_{\varepsilon\to0} \int_{\s^{n-1}}\int_{B_A}\int_0^1 & \frac{ \eps\left|Du(x)\cdot r \omega\right|^{p(x)+\eps}}{r^{p(x)+1}}\, dr\,dx\,d\mathcal{H}^{n-1}(\omega)\\
& = \liminf_{\varepsilon\to0} \int_{B_A}\left(\int_0^1 \eps \, r^{\eps-1}\,dr \int_{\s^{n-1}} \left|Du(x)\cdot \omega\right|^{p(x)+\eps}\, d\mathcal{H}^{n-1}(\omega)\right)dx \\
& = \liminf_{\varepsilon\to0} \int_{B_A}\left( \int_{\s^{n-1}} \left|Du(x)\cdot \omega\right|^{p(x)+\eps}\, d\mathcal{H}^{n-1}(\omega)\right)dx \\
& \ge \int_{B_A}\left( \int_{\s^{n-1}} \liminf_{\varepsilon\to0} \left|Du(x)\cdot \omega\right|^{p(x)+\eps}\, d\mathcal{H}^{n-1}(\omega)\right)dx \\
& = \int_{B_A}\left( \int_{\s^{n-1}} \left|Du(x)\cdot \omega\right|^{p(x)}\, d\mathcal{H}^{n-1}(\omega)\right)dx \\
& = \int_{B_A}p(x)\,K_{n,p(x)}\left|Du(x)\right|^{p(x)} \,dx ,
\end{aligned}
$$
hence the inequality
$$
\int_{B_A} p(x)\,K_{n,p(x)}\left|Du(x)\right|^{p(x)} \,dx \le \liminf_{\eps\to0} \int_{\s^{n-1}}\int_{B_A}\int_0^1\frac{\eps\,|u(x+r\omega)-u(x)|^{p(x)+\eps}}{r^{p(x)+1}}\,dr\,dx\,d\mathcal{H}^{n-1}(\omega) .
$$
By the arbitrariness of $A>0$, we conclude that  
$$
\int_{\R^n} p(x)\,K_{n,p(x)} \left|Du(x)\right|^{p(x)} \,dx \le \liminf_{\eps\to0} \int_{\R^n}\int_{\R^n} \frac{\eps\,|u(x)-u(y)|^{p(x)+\eps}}{|x-y|^{n+p(x)}}\,dx \, dy \le C(u) .
$$
Since $p^+\,K_{n,p^+} \leq p(x)\,K_{n,p(x)}$, if $C(u)$ is finite,
we get $u\in W^{1,p(\cdot)}\left(\R^n\right)$, concluding the proof.
\end{proof}

\begin{corollary}
Let $u\in L^{p(\cdot)}\left(\R^n\right)\cap C^2\left(\R^n\right)$ and set $\omega(x) = p(x)\,K_{n,p(x)}$. Then
\begin{equation}
\label{luxnorm}
\|\nabla u\|_{L^{p(\cdot)}\left(\R^n,\,\omega\right)}  \leq
\liminf_{\eps\to 0} \max_\pm \left(\int_{\R^n}\int_{\R^n}\frac{\eps\left|u(x)-u(y)\right|^{p(x)+\eps}}{|x-y|^{n+p(x)}}\,dx\,dy\right)^{1/p^\pm}
\end{equation}
and, if the right-hand side is finite, $u\in W^{1,p(\cdot)}\left(\R^n\right)$. In addition, if $p\in\mathcal{P}^{\log}\left(\R^n\right)$, we have $\mathcal{M}\left(\nabla u\right) \in L^{p(\cdot)}\left(\R^n\right)$.
\end{corollary}
\begin{proof}
By Lemma \ref{lemma4.4}, we have
$$
\int_{\R^n} \left|\nabla u(x)\right|^{p(x)} \omega(x) \, dx \le \liminf_{\varepsilon\to0} \int_{\R^n} \int_{\R^n} \frac{\eps\left|u(x)-u(y)\right|^{p(x)+\eps}}{|x-y|^{n+p(x)}} \, dx \, dy ,
$$
while, remembering property \eqref{weight}, we have also
$$
\begin{aligned}
\|\nabla u\|_{L^{p(\cdot)}\left(\R^n,\,\omega\right)} 
&\le  \max_\pm\left(\int_{\R^n} \left|\nabla u(x)\right|^{p(x)} \omega(x) \, dx \right)^{\frac{1}{p^\pm}} \\
&\le \max_\pm \liminf_{\varepsilon\to0} \left(\int_{\R^n} \int_{\R^n} \frac{\eps\left|u(x)-u(y)\right|^{p(x)+\eps}}{|x-y|^{n+p(x)}} \right)^{\frac{1}{p^\pm}} \\
&\le \liminf_{\eps\to0} \max_\pm \left(\int_{\R^n}\int_{\R^n}\frac{\eps\left|u(x)-u(y)\right|^{p(x)+\eps}}{|x-y|^{n+p(x)}}\,dx\,dy\right)^{\frac{1}{p^\pm}} ,
\end{aligned}
$$
so the first assertion follows. Since $p^+\,K_{n,p^+} \leq p(x)\,K_{n,p(x)}$, if the last term is finite,
we get $u\in W^{1,p(\cdot)}\left(\R^n\right)$.
Moreover, if $p\in\mathcal{P}^{\log}\left(\R^n\right)$, by applying Theorem \ref{maxfunctiongen}, there exists a constant $K_{p^-}>0$ such that
$$
\left\|\mathcal{M}\left(\nabla u\right)\right\|_{L^{p(\cdot)}\left(\R^n\right)} \le K_{p^-} \left\|\nabla u\right\|_{L^{p(\cdot)}\left(\R^n\right)},
$$
and the proof is complete.
\end{proof}

We will need the following lemma.

\begin{lemma}
	\label{lemma4.1}
	Let $\Omega\subset\R^n$ be a measurable set and let $\psi$ and $\phi$ be two non-negative measurable functions on the domain $\Omega\times\Omega$. If we take a measurable function $\alpha : \Omega\to(-1,+\infty)$, then
	$$
	\int_0^1 \underset{\phi(x,y)>\delta}{\int\int} \delta^{\alpha(x)} \psi(x,y) \, dx\, dy\,d\delta =\underset{\phi(x,y)\le1}{\int\int} \frac{\phi^{\alpha(x)+1}(x,y)}{\alpha(x)+1}\psi(x,y)\,dx\, dy + \underset{\phi(x,y)>1}{\int\int} \frac{\psi(x,y)}{\alpha(x)+1} \,dx\, dy.
	$$
\end{lemma}
\begin{proof}
	From a direct computation, by using Fubini-Tonelli's theorem, we have
	$$
	\begin{aligned}
	&\int_0^1 \underset{\phi(x,y)>\delta}{\int\int} \delta^{\alpha(x)} \psi(x,y) \, dx\,dy\,d\delta = 
	\int_{\Omega} \int_{\Omega} \psi(x,y) \underset{\phi(x,y)>\delta}{\int_0^1} \delta^{\alpha(x)} \,d\delta \,dx\,dy\\
	& = \int_{\Omega} \int_{\Omega} \psi(x,y) \underset{\delta<\phi(x,y)\le1}{\int_0^1} \delta^{\alpha(x)} \,d\delta \,dx\,dy + \int_{\Omega} \int_{\Omega} \psi(x,y) \underset{\phi(x,y)>1}{\int_0^1} \delta^{\alpha(x)} \,d\delta \,dx\,dy \\
	& = \underset{\phi(x,y)\le1}{\int\int} \psi(x,y) \left[\frac{\delta^{\alpha(x)+1}}{\alpha(x)+1}\right]_0^{\phi(x,y)} \,dx\,dy + \underset{\phi(x,y)>1}{\int\int} \psi(x,y) \left[\frac{\delta^{\alpha(x)+1}}{\alpha(x)+1}\right]_0^1 \,dx\,dy \\
	& = \underset{\phi(x,y)\le1}{\int\int} \psi(x,y) \frac{\phi^{\alpha(x)+1}(x,y)}{\alpha(x)+1} \,dx\,dy + \underset{\phi(x,y)>1}{\int\int} \frac{1}{\alpha(x)+1}\psi(x,y) \,dx\,dy ,
	\end{aligned}
	$$
	and the assertion follows.
\end{proof}

\begin{theorem}
The following facts hold:
\begin{enumerate}
\item[(a)] for every $u\in W^{1,p^{\pm}}\left(\R^n\right)$, there exists $C>0$, depending only on $n$ and $p^\pm$, such that
$$
\begin{aligned}
\sup_{0<\varepsilon<1} & \underset{|u(x)-u(y)|\le1}{\int_{\R^n}\int_{\R^n}} \frac{\eps\,|u(x)-u(y)|^{p(x)+\eps}}{|x-y|^{n+p(x)}} \, dx \, dy \\
+ & \underset{|u(x)-u(y)|>1}{\int_{\R^n}\int_{\R^n}} \frac{1}{|x-y|^{n+p(x)}} \, dx \, dy \le C \left(\|\nabla u\|_{L^{p^+}\left(\R^n\right)}^{p^+}+\|\nabla u\|_{L^{p^-}\left(\R^n\right)}^{p^-}\right) ;
\end{aligned}
$$

\item[(b)] for every $u\in W^{1,p^{\pm}}\left(\R^n\right)$, 
$$
\lim_{\eps\to0} \underset{|u(x)-u(y)|\le1}{\int_{\R^n}\int_{\R^n}} \frac{\eps\,|u(x)-u(y)|^{p(x)+\eps}}{|x-y|^{n+p(x)}} \, dx \, dy = \int_{\R^n} p(x)\,K_{n,p(x)} |\nabla u(x)|^{p(x)} \, dx.
$$
\end{enumerate}
\end{theorem}
\begin{proof}
\item[(a)] Let $u\in W^{1,p^{\pm}}\left(\R^n\right)$. By Lemma \ref{lemma4.2}, there exists $C>0$, depending only on $n$ and $p^\pm$, such that
$$
\underset{\left|u(x)-u(y)\right|>\delta} {\int_{\R^n}\int_{\R^n}} \frac{\delta^{p(x)}}{|x-y|^{n+p(x)}} \, dx \, dy \le C \left(\|\nabla u\|_{L^{p^+}\left(\R^n\right)}^{p^+}+\|\nabla u\|_{L^{p^-}\left(\R^n\right)}^{p^-}\right) ,
$$
for all $\delta>0$. 
Multiplying this inequation by $\eps\delta^{\eps-1}$, with $\eps\in(0,1)$, and integrating with respect to $\delta$ over the interval $(0,1)$, we have
\begin{equation}
\label{miserve2}
\int_0^1 \underset{\left|u(x)-u(y)\right|>\delta} {\int_{\R^n}\int_{\R^n}} \frac{\eps\delta^{p(x)+\eps-1}}{|x-y|^{n+p(x)}} \, dx \, dy \, d\delta
\le C \left(\|\nabla u\|_{L^{p^+}\left(\R^n\right)}^{p^+}+\|\nabla u\|_{L^{p^-}\left(\R^n\right)}^{p^-}\right).
\end{equation}
By using Lemma \ref{lemma4.1} with 
$$
\alpha(x) = p(x)+\eps-1, \qquad \phi(x,y)=|u(x)-u(y)| \qquad \text{and} \qquad \psi(x,y)=\dfrac{\eps}{|x-y|^{n+p(x)}},
$$
the integral at the left-hand side of the inequality becomes
$$
\begin{aligned}
\int_0^1 \underset{\left|u(x)-u(y)\right|>\delta} {\int_{\R^n}\int_{\R^n}} \frac{\eps\delta^{p(x)+\eps-1}}{|x-y|^{n+p(x)}} \, dx \, dy & = \underset{\left|u(x)-u(y)\right|\le1} {\int_{\R^n}\int_{\R^n}} \frac{\eps\left|u(x)-u(y)\right|^{p(x)+\eps}}{(p(x)+\eps)|x-y|^{n+p(x)}} \, dx \, dy \\
& + \underset{\left|u(x)-u(y)\right|>1} {\int_{\R^n}\int_{\R^n}} \frac{\eps}{(p(x)+\eps)|x-y|^{n+p(x)}} \, dx \, dy .
\end{aligned}
$$
In particular, by equation \eqref{miserve2}, it follows
\begin{equation*}
\sup_{0<\eps<1}\underset{\left|u(x)-u(y)\right|\le1} {\int_{\R^n}\int_{\R^n}} \frac{\eps\left|u(x)-u(y)\right|^{p(x)+\eps}}{|x-y|^{n+p(x)}} \, dx \, dy \le C\,(p^++1) \left(\|\nabla u\|_{L^{p^+}\left(\R^n\right)}^{p^+}+\|\nabla u\|_{L^{p^-}\left(\R^n\right)}^{p^-}\right).
\end{equation*}
Finally, by matching this formula with the starting inequality for $\delta=1$, it follows
$$
\begin{aligned}
\sup_{0<\varepsilon<1} & \underset{|u(x)-u(y)|\le1}{\int_{\R^n}\int_{\R^n}} \frac{\eps\,|u(x)-u(y)|^{p(x)+\eps}}{|x-y|^{n+p(x)}}\, dx \, dy \\
+ & \underset{|u(x)-u(y)|>1}{\int_{\R^n}\int_{\R^n}} \frac{1}{|x-y|^{n+p(x)}} \, dx \, dy \le C \left(\|\nabla u\|_{L^{p^+}\left(\R^n\right)}^{p^+}+\|\nabla u\|_{L^{p^-}\left(\R^n\right)}^{p^-}\right) ,
\end{aligned}
$$
for some constant $C$ depending only on $n$, $p^\pm$ and the first assertion follows.

\item[(b)] Let $u\in W^{1,p^{\pm}}\left(\R^n\right)$. 
Let us compute the limit
$$
\lim_{\eps\to0} \int_0^1 \eps\,\delta^{\eps-1} \underset{|u(x)-u(y)|>\delta}{\int_{\R^n}\int_{\R^n}}  \frac{(p(x)+\eps)\delta^{p(x)}}{|x-y|^{n+p(x)}} \,dx\,dy\,d\delta .
$$
Taking $\tau\in(0,1)$, we can write the integral with respect to $\delta$ as
$$
\begin{aligned}
\lim_{\eps\to0^+} \int_0^1 \eps\,\delta^{\eps-1} \underset{|u(x)-u(y)|>\delta}{\int_{\R^n}\int_{\R^n}} & \frac{(p(x)+\eps)\delta^{p(x)}}{|x-y|^{n+p(x)}} \,dx\,dy\,d\delta \\
& =\lim_{\tau\to0^+}\lim_{\eps\to0^+} \int_0^{\tau} \eps\,\delta^{\eps-1} \underset{|u(x)-u(y)|>\delta}{\int_{\R^n}\int_{\R^n}}  \frac{(p(x)+\eps)\delta^{p(x)}}{|x-y|^{n+p(x)}} \,dx\,dy\,d\delta\\
&+\lim_{\tau\to0^+}\lim_{\eps\to0^+} \int_{\tau}^1 \eps\,\delta^{\eps-1} \underset{|u(x)-u(y)|>\delta}{\int_{\R^n}\int_{\R^n}}  \frac{(p(x)+\eps)\delta^{p(x)}}{|x-y|^{n+p(x)}} \,dx\,dy\,d\delta.
\end{aligned}
$$
On the one hand, the second integral at the right-hand side goes to $0$ as $\eps\to0$, since
$$
\begin{aligned}
& \lim_{\eps\to0^+} \int_{\tau}^1 \eps\,\delta^{\eps-1} \underset{|u(x)-u(y)|>\delta}{\int_{\R^n}\int_{\R^n}}  \frac{(p(x)+\eps)\delta^{p(x)}}{|x-y|^{n+p(x)}} \,dx\,dy\,d\delta \\
&\le \lim_{\eps\to0^+} \int_{\tau}^1 \eps\,\delta^{\eps-1} \underset{|u(x)-u(y)|>\delta}{\int_{\R^n}\int_{\R^n}}  \frac{(p^++\eps)\delta^{p(x)}}{|x-y|^{n+p(x)}} \,dx\,dy\,d\delta \\
& \le
 {C}_{n,p^{\pm}} \left(\|\nabla u\|_{L^{p^+}\left(\R^n\right)}^{p^+}+\|\nabla u\|_{L^{p^-}\left(\R^n\right)}^{p^-}\right) \lim_{\eps\to0}\, [(p^++\eps)(1-\tau^\eps)]=0.
\end{aligned}
$$
On the other hand, we can write the first integral by making the change of variable $\delta=\tau z$, 
$$
\begin{aligned}
\int_0^{\tau} \eps\,\delta^{\eps-1} & \underset{|u(x)-u(y)|>\delta}{\int_{\R^n}\int_{\R^n}}  \frac{(p(x)+\eps)\delta^{p(x)}}{|x-y|^{n+p(x)}} \,dx\,dy\,d\delta \\
& = \int_0^1 \eps\,\tau^\eps z^{\eps-1} \underset{|u(x)-u(y)|>\tau z}{\int_{\R^n}\int_{\R^n}} \frac{(p(x)+\eps)(\tau z)^{p(x)}}{|x-y|^{n+p(x)}}\,dx\,dy\, dz .
\end{aligned}
$$
From the arbitrariness of $\tau\in(0,1)$, making the limit as $\tau\to0^+$ and remembering both the anisotropic limit formulas, it happens that
$$
\begin{aligned}
&\lim_{\tau\to0^+}\lim_{\eps\to0^+}  \int_0^{\tau} \eps\,\delta^{\eps-1} \underset{|u(x)-u(y)|>\delta}{\int_{\R^n}\int_{\R^n}}  \frac{(p(x)+\eps)\delta^{p(x)}}{|x-y|^{n+p(x)}} \,dx\,dy\,d\delta \\
&= \lim_{\tau\to0^+}\lim_{\eps\to0^+}   \tau^\eps \int_0^1 \eps\, z^{\eps-1} \underset{|u(x)-u(y)|>\tau z}{\int_{\R^n}\int_{\R^n}}  \left(\frac{p(x)\,(\tau z)^{p(x)}}{|x-y|^{n+p(x)}} +\eps \, \frac{(\tau z)^{p(x)}}{|x-y|^{n+p(x)}}\right)\,dx\,dy\, dz \\
&= \lim_{\tau\to0^+}\lim_{\eps\to0^+}    \int_0^1 \eps\, z^{\eps-1} \underset{|u(x)-u(y)|>\tau z}{\int_{\R^n}\int_{\R^n}} \frac{p(x)\,(\tau z)^{p(x)}}{|x-y|^{n+p(x)}} \,dx\,dy\, dz \\
&= \int_{\R^n} p(x)\,K_{n,p(x)}|\nabla u(x)|^{p(x)}\,dx.
\end{aligned}
$$
In turn, we can conclude that
$$
\lim_{\eps\to0} \int_0^1 \underset{|u(x)-u(y)|>\delta}{\int_{\R^n}\int_{\R^n}} \frac{(p(x)+\eps)\,\eps\,\delta^{p(x)+\eps-1}}{|x-y|^{n+p(x)}} \,dx\,dy\,d\delta = \int_{\R^n} p(x)\,K_{n,p(x)} |\nabla u(x)|^{p(x)}\,dx .
$$
Applying Lemma \ref{lemma4.1} with
$$
\alpha(x) = p(x)+\eps-1, \qquad \phi(x,y) =|u(x)-u(y)| \qquad \text{and} \qquad \psi(x,y) =\dfrac{(p(x)+\eps)\,\eps}{|x-y|^{n+p(x)}}, 
$$
we write
$$
\begin{aligned}
\int_0^1 \underset{|u(x)-u(y)|>\delta}{\int_{\R^n}\int_{\R^n}} \frac{(p(x)+\eps)\,\eps\,\delta^{p(x)+\eps-1}}{|x-y|^{n+p(x)}} \,dx\,dy\,d\delta & = \underset{\left|u(x)-u(y)\right|\le1} {\int_{\R^n}\int_{\R^n}} \frac{\eps\left|u(x)-u(y)\right|^{p(x)+\eps}}{|x-y|^{n+p(x)}} \, dx \, dy \\
&+\underset{\left|u(x)-u(y)\right|>1} {\int_{\R^n}\int_{\R^n}} \frac{\eps}{|x-y|^{n+p(x)}} \, dx \, dy.
\end{aligned}
$$
Hence, by taking the limit as $\eps\to0$, the assertion follows.
\end{proof}

\begin{theorem}
The following facts hold:
\begin{enumerate}
\item[(a)] if $u\in L^{p(\cdot)}\left(\R^n\right) \cap C^2_b\left(\R^n\right)$ is a bounded $C^2$ function and
$$
\sup_{0<\varepsilon<1} \underset{|u(x)-u(y)|\le1}{\int_{\R^n}\int_{\R^n}} \frac{\eps\,|u(x)-u(y)|^{p(x)+\eps}}{|x-y|^{n+p(x)}}\, dx \, dy + \underset{|u(x)-u(y)|>1}{\int_{\R^n}\int_{\R^n}} \frac{1}{|x-y|^{n+p(x)}} \, dx \, dy<+\infty ,
$$
then $u\in W^{1,p(\cdot)}\left(\R^n\right)$;

\item[(b)] if $u\in L^{p(\cdot)}\left(\R^n\right)$ and 
$$
\sup_{0<\delta<1}  \underset{|u(x)-u(y)|>\delta}{\int_{\R^n}\int_{\R^n}} \frac{\delta^{p(x)}}{|x-y|^{n+p(x)}}\, dx \, dy <+\infty ,
$$
then $u\in W^{1,p(\cdot)}\left(\R^n\right)$.
\end{enumerate}
\end{theorem}
\begin{proof}
\item[(a)] Let $u\in L^{p(\cdot)}\left(\R^n\right)$ be a bounded $C^2$ function satisfying (a). Then
$$
\begin{aligned}
\sup_{0<\varepsilon<1} {\int_{\R^n}\int_{\R^n}} & \frac{\eps\,|u(x)-u(y)|^{p(x)+\eps}}{|x-y|^{n+p(x)}}\, dx \, dy \\
& \le \sup_{0<\varepsilon<1}  \underset{|u(x)-u(y)|\le1}{\int_{\R^n}\int_{\R^n}} \frac{\eps\,|u(x)-u(y)|^{p(x)+\eps}}{|x-y|^{n+p(x)}}\, dx \, dy \\ 
& + 2^{p^++1}\max\left\{1,\|u\|_{L^\infty\left(\R^n\right)}^{p^++1}\right\} \underset{|u(x)-u(y)|>1}{\int_{\R^n}\int_{\R^n}} \frac{1}{|x-y|^{n+p(x)}}\, dx \, dy < +\infty .
\end{aligned}
$$
The assertion follows from Lemma \ref{lemma4.4}.
\vskip4pt
\item[(b)] Let $u\in L^{p(\cdot)}\left(\R^n\right) \cap C^2_b\left(\R^n\right)$ such that
\begin{equation}
\label{qsta}
\sup_{0<\delta<1} \underset{|u(x)-u(y)|>\delta}{\int_{\R^n}\int_{\R^n}} \frac{\delta^{p(x)}}{|x-y|^{n+p(x)}}\, dx \, dy \le C <+\infty ,
\end{equation}
for some constant $C>0$. 
Multiplying the inequality by $\eps\delta^{\eps-1}$, with $\eps\in(0,1)$, and integrating with respect to $\delta$ over the interval $(0,1)$, we have
$$
\int_0^1\underset{|u(x)-u(y)|>\delta}{\int_{\R^n}\int_{\R^n}} \frac{\eps\delta^{p(x)+\eps-1}}{|x-y|^{n+p(x)}}\, dx \, dy\,d\delta \le C .
$$
Applying Lemma \ref{lemma4.1} with
$$
\alpha(x) = p(x)+\eps-1, \qquad \phi(x,y)=|u(x)-u(y)| \qquad \text{and} \qquad \psi(x,y)=\dfrac{\eps}{|x-y|^{n+p(x)}},
$$
the left-hand side of the last inequality becomes
$$
\underset{\left|u(x)-u(y)\right|\le1} {\int_{\R^n}\int_{\R^n}} \frac{\eps\left|u(x)-u(y)\right|^{p(x)+\eps}}{(p(x)+\eps)|x-y|^{n+p(x)}} \, dx \, dy  + \underset{\left|u(x)-u(y)\right|>1} {\int_{\R^n}\int_{\R^n}} \frac{\eps}{(p(x)+\eps)|x-y|^{n+p(x)}} \, dx \, dy\le C ,
$$
so we get
$$
\begin{aligned}
\sup_{\eps\in(0,1)} \underset{\left|u(x)-u(y)\right|\le1} {\int_{\R^n}\int_{\R^n}} \frac{\eps\left|u(x)-u(y)\right|^{p(x)+\eps}}{|x-y|^{n+p(x)}}\le C(p^++1) .
\end{aligned}
$$
Recalling that, by \eqref{qsta},
$$
\underset{|u(x)-u(y)|>1}{\int_{\R^n}\int_{\R^n}} \frac{1}{|x-y|^{n+p(x)}}\, dx \, dy \le C,
$$
assumption (a) of the Theorem is satisfied, so $u\in W^{1,p(\cdot)}\left(\R^n\right)$ and the proof is complete.
In the general case, one can argue as in \cite{NPSV} by using a density argument.
\end{proof}

We conclude with an observation dealing with the limiting case $p^-=1$.
\begin{remark}
	\label{remp1}
	Let $u \in L^{p(\cdot)}\left(\R^n\right)$ with $1=p^- \le p^+<+\infty$
	and set $E=\left\{x\in\R^n: p(x)=1\right\}$. Assume ${\rm int}(E)\neq \emptyset$ and
	$$
	\Lambda=\sup_{0<\delta<1}  \underset{|u(x)-u(y)|>\delta}{\int_{\R^n}\int_{\R^n}} \frac{\delta^{p(x)}}{|x-y|^{n+p(x)}}\, dx \, dy <+\infty.
	$$
	Then
	$$
	\sup_{B\subset {\rm int}(E)} |B|^{-\frac{n+1}{n}}\int_B\int_B |u(x)-u(y)|\,dx\,dy<+\infty.
	$$
	In particular, for $n=1$, it reads as
	$$
	\sup_{B\subset {\rm int}(E)} \dashint_B\dashint_B |u(x)-u(y)|\,dx\,dy<+\infty,
	$$
	so $u\in BMO(E)$, the space of bounded mean oscillation functions on $E$.
	In fact, let $x_0\in {\rm int}(E)$ and $B\subset {\rm int}(E)$
	a ball centered at $x_0$. 
	We have, for all $\delta\in(0,1)$,
	$$
	\underset{|u(x)-u(y)|>\delta}{\int_{B}\int_{B}} \frac{\delta}{|x-y|^{n+1}}\, dx \, dy \leq    \underset{|u(x)-u(y)|>\delta}{\int_{\R^n}\int_{\R^n}} \frac{\delta^{p(x)}}{|x-y|^{n+p(x)}}\, dx \, dy
	\leq \Lambda .
	$$
	Then, taking into account \cite[(a) of Theorem 1]{Ng11}, we have 
	\begin{align*}
	\int_B\int_B |u(x)-u(y)|\,dx\,dy &\leq C\left(|B|^{\frac{n+1}{n}}\! \underset{|u(x)-u(y)|>\delta}{\int_{B}\int_{B}} \frac{\delta}{|x-y|^{n+1}}\, dx \, dy+\delta|B|^2  \right) \\
	&\leq C\left(\Lambda |B|^{\frac{n+1}{n}}+\delta|B|^2  \right) ,
	\end{align*}
	for some constant $C$ depending on $n$. Then the above assertions follow by the arbitrariness of $\delta\in(0,1)$ -- making the limit as $\delta\to0^+$ -- and $B\subset {\rm int}(E)$.
\end{remark}

\end{document}